\documentclass[10pt, reqno]{amsart} 

\usepackage{amsmath, amsthm, amssymb, amsfonts}
\usepackage{mathrsfs}
\usepackage{enumerate}
\usepackage{enumitem}
\usepackage{graphicx}
\usepackage{caption, subcaption}
\usepackage{float}
\usepackage{mathtools}

\usepackage[colorlinks=true, linkcolor=blue, citecolor=red, urlcolor=blue]{hyperref}

\usepackage{lineno} 

\usepackage[colorinlistoftodos, textsize=footnotesize]{todonotes} 
\usepackage{orcidlink}

\theoremstyle{plain}
\newtheorem{theorem}{Theorem}[section]
\newtheorem{corollary}[theorem]{Corollary}
\newtheorem{lemma}[theorem]{Lemma}
\newtheorem{proposition}[theorem]{Proposition}

\theoremstyle{definition}

\theoremstyle{remark}
\newtheorem{remark}[theorem]{Remark}

\newcommand{\R}{\mathbb{R}}

\allowdisplaybreaks

\title[Bounds on the eigenvalues of principal submatrices]{Aggregate Bounds on the eigenvalues of the principal submatrices of a Hermitian matrix and majorization relations}

\author{Hristo Sendov \orcidlink{0000-0002-0908-1535}}
\address{Department of Statistical and Actuarial Sciences \\
Department of Mathematics \\
Western University \\
1151 Richmond Street \\
London, ON, N6A 5B7 Canada}
\email{hsendov@uwo.ca}
\thanks{The first author was partially supported by the Natural Sciences and Engineering Research Council (NSERC) of Canada. (Grant number RGPIN-2020-06425.)}

\author{Mengxu Yuan \orcidlink{0009-0002-5166-183X}}
\address{Department of Mathematics \\
Western University \\
1151 Richmond Street \\
London, ON, N6A 5B7 Canada}
\email{myuan89@uwo.ca}


\begin{document}

\begin{abstract}
We extend bounds, proved by R.C.~Thompson in 1966, on the sum of the $j$-th largest eigenvalues of the $(n-1) \times (n-1)$ principal matrices of an $n \times n$ Hermitian matrix. Our bounds are stronger than just summing up Thompson's bounds. We achieve the extensions as a corollary of a more general result giving bounds on the zeros of the generalized derivatives of polynomials with real roots. We use the extended bounds to obtain majorization relationships between the eigenvalues of all $m \times m$ principal matrices of an $n \times n$ Hermitian matrix. These majorization relationships imply both a well-known majorization result by Schur and the well-known Szasz's inequalities. 
\end{abstract}

\maketitle

\section{Introduction and overview} 
\label{sec:intro}

The relationships between the eigenvalues of $n \times n$ Hermitian matrix and those of its $(n-1) \times (n-1)$ principal submatrices resemble those between the zeros and critical points of hyperbolic polynomials.  The Cauchy’s interlacing theorem states that real numbers satisfy the interlacing inequalities
$$
\lambda_1 \ge \mu_1 \ge \lambda_2 \ge \mu_2 \ge \cdots \ge \mu_{n-1} \ge \lambda_n
$$
if and only if there is an $n \times n$ Hermitian matrix with eigenvalues $\lambda_1,\ldots, \lambda_n$ and having an  $(n-1) \times (n-1)$ principal submatrix with eigenvalues $\mu_1,\ldots, \mu_{n-1}$, see Section 4.3 in \cite{horn2013}. Many improvements and extensions of this basic fact have been established since. To state the most pertinent results, we fix the notation. 
The vector of eigenvalues of an $n \times n$ Hermitian matrix $A$ is denoted by 
$\lambda(A) = (\lambda_1(A),  \dots, \lambda_n(A))$ and is assumed to be ordered non-increasingly
$$
\lambda_1(A) \ge  \dots \ge \lambda_n(A).
$$
For each $i \in \{1, \dots, n\}$, let $A_i$ denote the 
$(n-1) \times (n-1)$ principal submatrix obtained by deleting the $i$-th row and column of $A$. The vector of eigenvalues of $A_i$ is denoted by $\mu(A_i) = (\mu_{i,1},  \dots, \mu_{i,n-1})$ and is also ordered non-increasingly
\[
\mu_{i,1} \ge  \cdots \ge \mu_{i,n-1}.
\]
In \cite{thompson1966}, Thompson proved the following bounds for the aggregate behaviour of the $j$-th eigenvalues of the $(n-1) \times (n-1)$ principal submatrices.

\begin{theorem}[Thompson] 
\label{thm:spectral_bounds}
For any $j$,  $1\le j \le n-1$, we have
\begin{equation}
\label{2026-01-09-ineq}
   \lambda_j +  (n-1)\lambda_{j+1}  \le \sum_{i=1}^n \mu_{i, j} \le (n-1)\lambda_j + \lambda_{j+1}.
\end{equation}
\end{theorem} 
The algebraic relation that provided the theoretical foundation of Thompson's result is the following eigenvector-eigenvalue identity. For the history and various derivations of this identity one should refer to the recent work \cite{denton2022}.

\begin{theorem}
\label{thm:eigenvector_identity}
Let $v_i$ be a normalized eigenvector associated with $\lambda_i$. Then 
\begin{equation}
    |v_{i,j}|^2 \prod_{\substack{k=1 \\ k \ne i}}^n (\lambda_i - \lambda_k) = \prod_{k=1}^{n-1} (\lambda_i - \mu_{j,k}),
\end{equation}
where $v_{i,j}$ denotes the $j$-th component of $v_i$.
\end{theorem}

In subsequent papers, see \cite{thompson1968a} and \cite{thompson1968b}, Thompson investigated further consequences of the eigenvector-eigenvalue identity. Another notable relationship is given by Theorem 3 in \cite{johnson1981} as follows. 

\begin{theorem}[Johnson \& Robinson]
\label{2026-01-15-JR}
For each $j$, $1 \le j \le n-1$, we have 
$$
\max_{1 \le i \le n} \mu_{i,j} \ge \frac{n-j}{n} \lambda_{j} +  \frac{j}{n} \lambda_n
\quad \mbox{and} \quad
\min_{1 \le i \le n} \mu_{i,j} \le \frac{n-j}{n} \lambda_{1} + \frac{j}{n} \lambda_{j+1}.
$$
\end{theorem}

The main goal of this work is to extend the bounds in Theorem~\ref{thm:spectral_bounds} to bounds for sums of the form $\sum_{j=\ell}^r \sum_{i=1}^n \mu_{i,j}$ for any integers $1 \le \ell \le r \le n-1$. Of course, it is straightforward to add the inequalities \eqref{2026-01-09-ineq} for $j=\ell, \ldots, r$, but the thus obtained bounds are too weak for the applications that we have in mind. Instead, we derive the following result. 

\begin{theorem}
\label{thm:aggregate_bound-intro}
For any $\ell$ and $r$, $1 \le \ell \le r \le n-1$, the following inequalities hold:
\begin{align*}
 (r-\ell+1)\lambda_{\ell} + (n-1)\sum_{j=\ell}^r \lambda_{j+1}  \le \sum_{k=1}^n \sum_{j=\ell}^r   \mu_{k,j} &\le (n-1)\sum_{j=\ell}^r \lambda_j + (r-\ell+1)\lambda_{r+1}.
\end{align*}
\end{theorem}

Applying the bounds in Theorem~\ref{thm:aggregate_bound-intro} with $\ell=r=j$, one recovers the bounds in Theorem~\ref{thm:spectral_bounds}.
A corollary of this result is the following.

\begin{corollary}
\label{2026-01-15-cor}
For any $\ell$ and $r$, $1 \le \ell \le r \le n-1$, the following inequalities hold:
$$
(n-\ell)\lambda_\ell  + (n-1)\sum_{j=\ell+1}^{r} \lambda_j + r\lambda_n
\le \sum_{k=1}^n \sum_{j=\ell}^r  \mu_{k,j} 
\le (n-\ell)\lambda_1 + (n-1)\sum_{j=\ell+1}^{r} \lambda_j + r\lambda_{r+1} .
$$
\end{corollary}

Note that by taking $\ell = r = j$ in Corollary~\ref{2026-01-15-cor}, we obtain
$$
(n-j)\lambda_j + j\lambda_n  \le \sum_{k=1}^n \mu_{k,j} 
\le (n-j)\lambda_1 + j\lambda_{j+1},
$$
which is a stronger statement than Theorem~\ref{2026-01-15-JR}.

The bounds in Theorem~\ref{thm:aggregate_bound-intro} are a consequence of the following theorem which is one of the main results in this paper.

\begin{theorem}
\label{thm:main_weighted}
Let $\lambda_1  \ge \dots \ge \lambda_n$ be real numbers and let $w_1, \dots, w_n$ be non-negative weights with sum $1$.
Let $\mu_1 \ge \dots \ge \mu_{n-1}$ be the (real) roots of the polynomial
\begin{align*}
p(x):= \sum_{i=1}^n w_i \prod_{\substack{j=1 \\ j \neq i}}^n (\lambda_j - x)
\end{align*}
Then for any $\ell$ and $r$, $1 \le \ell \le r \le n-1$, the following inequalities hold:
\begin{align*} 
\sum_{j=\ell}^r \lambda_{j+1} + \sum_{j=\ell}^r \frac{w_{j+1}}{L_{r+1}} (\lambda_\ell - \lambda_{j+1})  \le  \sum_{j=\ell}^r \mu_j \le \sum_{j=\ell}^r \lambda_j - \sum_{j=\ell}^r \frac{w_j}{U_{\ell}} (\lambda_j - \lambda_{r+1}),
\end{align*}
provided that $L_{r+1}$ and $U_{\ell}$ are non-zero,  where
$$
U_{\ell} := \sum_{i=\ell}^n w_i \quad \mbox{and} \quad L_{r+1} := \sum_{i=1}^{r+1} w_i.
$$\end{theorem}

The notion of majorization between vectors in $\R^N$ is a powerful tool for capturing relationships between zeros and critical points of polynomials, see \cite{katsoprinakis:2007}. Similarly, it plays a big role in describing relationships between the eigenvalues of Hermitian matrices and those of its principal submatrices. 

For $x \in \mathbb{R}^N$, let $x^{\downarrow}$ denote the vector with the same entries as $x$, but sorted in non-increasing order. (Similarly, $x^{\uparrow}$ denotes the vector with the same entries as $x$, but sorted in non-decreasing order.)

For $x, y \in \mathbb{R}^N$, we say that $x$ {\it majorizes} $y$, denoted as $x \succ y$, if the following inequalities hold:
\begin{equation}
    \sum_{i=1}^k x^{\downarrow}_i \ge \sum_{i=1}^k y^{\downarrow}_i, \quad \text{for all } k = 1, \dots, N,
\end{equation}
and the last inequality holding with equality. A matrix $D \in \mathbb{R}^{N \times N}$ is called {\it doubly stochastic} if all its entries are non-negative and every row and every column sums to $1$. 
For more information about majorization, see \cite{Marchall:2011}. 

The classical example of majorization in matrix theory is a result by Schur, see Exercise~II.1.12 in \cite{bhatia1997}. It states that for any $n \times n$ Hermitian matrix $A$, 
\begin{align}
\label{2026-01-13-schur}
\lambda(A) \succ \mbox{diag\,}(A).
\end{align}
In fact, a necessary and sufficient condition for there to be $n \times n$ Hermitian matrix with eigenvalues $x$ and diagonal $y$ is that $x \succ y$, see Theorem 4.3.48 in \cite{horn2013}. 

Other classical majorization relationships are the results by Fan and Lidskii, see Theorem 4.3.47 in \cite{horn2013}. Combined they state that for any two $n \times n$ Hermitian matrices $A$ and $B$, we have
$$
\lambda(A) + \lambda(B) \succ \lambda(A+B) \succ \lambda(A) + \lambda^{\uparrow}(B).
$$



Let $X_m(A)$, $1 \le m \le n$, be the vector containing the eigenvalues of all $m \times m$ principal submatrices of 
$A$, counting repetitions.  As a consequence of Theorem~\ref{thm:aggregate_bound-intro} we obtain the second main result of the paper.
\begin{theorem}[Spectral Hierarchy]
Let $A$ be an $n \times n$ Hermitian matrix. For any integers $1 \le k \le m \le n$, we have
\begin{align} 
\label{ab:eq:general_hierarchy}
    \underbrace{X_m(A) \cup \dots \cup X_m(A)}_{\binom{m-1}{k-1} \text{ times}} \succ \underbrace{X_k(A) \cup \dots \cup X_k(A)}_{\binom{n-k}{m-k} \text{ times}}.
\end{align}
\end{theorem}

Here, we abuse the union notation to denote the concatenation of vectors.  On the one hand, when $k=1$ and $m=n$, majorization \eqref{ab:eq:general_hierarchy} recovers \eqref{2026-01-13-schur}. On the other hand,  let $P_m(A)$ denote the product of all $m \times m$ principal minors of $A$, $1 \le m \le n$.  Example~II.3.16 in \cite{bhatia1997} shows that if $x \succ y$, where $x, y \in \mathbb{R}^N_+$, then $\prod_{i=1}^N x_i \le \prod_{i=1}^N y_i$.
Thus, for a positive semidefinite matrix $A$, majorization~\eqref{ab:eq:general_hierarchy} implies that
$$
P_{m}(A)^{\binom{m-1}{k-1}} \le P_k(A)^{\binom{n-k}{m-k}}, \mbox{ for all } m = 1,\ldots, n-1.
$$
Raising both sides to the power $1/Z$, where 
$$
Z:=\binom{m-1}{k-1}\binom{n-1}{m-1} = \binom{n-k}{m-k}\binom{n-1}{k-1},
$$ 
leads to 
\begin{equation}
    P_{m}(A)^{1/\binom{n-1}{m-1}} \le P_k(A)^{1/\binom{n-1}{k-1}}, \mbox{ for all } m = 1,\ldots, n-1.
\end{equation}
These are the well-known Szasz's inequalities, see \cite[Theorem 7.8.11]{horn2013}.

%
%
%


\section{The main result} 
\label{sec:3}

This section is dedicated to the proof of Theorem~\ref{thm:renormalized_bounds}. It gives bounds on the sum of several consecutive zeros of a generalized derivative of a polynomial with real roots. In Corollary~\ref{2026-01-14-cor}, Proposition~\ref{lem:equal_weights}, and Corollary~\ref{cor:relaxed_equal}, we give relaxations of the bounds and particular cases, that may be more useful in practice. At the end, we interpret the bounds in Corollary~\ref{cor:relaxed_equal} as majorization relationships.

The following lemma is standard.
\begin{lemma}
Let $\lambda_1 \ge \dots \ge \lambda_n$ 
be real numbers (poles) and let 
$$
w_1, \dots, w_n \mbox{ be non-negative weights with } \sum_{i=1}^n w_i = 1. 
$$
The polynomial 
\begin{equation}
    \label{eq:rational_eq}
p(x) :=   \sum_{i=1}^n w_i \prod_{\substack{j=1 \\ j \neq i}}^n (x-\lambda_j)
\end{equation}
has $(n-1)$ real roots, $\mu_1 \ge \dots \ge \mu_{n-1}$, that interlace with the poles
\begin{align}
\label{2026-01-09-strictInterl}
    \lambda_1 \ge \mu_1 \ge \lambda_2 \ge \mu_2 \ge \dots \ge \mu_{n-1} \ge \lambda_n.
\end{align}
If the poles are distinct, then the interlacing is strict.
\end{lemma}

\begin{proof}
Since the weights sum up to $1$, the degree of $p(x)$ is $n-1$.
Assume first that the poles are distinct and the weights are positive. Consider the rational function
\begin{align}
\label{2016-01-13-rat}
f(x) := \sum_{i=1}^n \frac{w_i}{x-\lambda_i} = \frac{p(x)}{\prod_{j=1}^n (x-\lambda_j)}.
\end{align}
It is strictly decreasing on each interval $(\lambda_{i+1}, \lambda_i)$ with $f(\lambda_{i+1}+)=\infty$ and $f(\lambda_{i}-)=-\infty$.  Thus, $f(x)$ vanishes exactly once in each such interval. These vanishing points are zeros of $p(x)$ and they strictly interlace with the poles. 

The general case follows using the fact that zeros of polynomials are continuous functions of its coefficients. 
\end{proof}

\begin{theorem}[The main result]
\label{thm:renormalized_bounds}
For any indices $\ell$ and $r$, $1 \le \ell \le r \le n-1$, we have the inequalities:
\begin{align*} 
\sum_{j=\ell}^r \lambda_{j+1} + \sum_{j=\ell}^r \frac{w_{j+1}}{L_{r+1}} (\lambda_\ell - \lambda_{j+1})  \le  \sum_{j=\ell}^r \mu_j \le \sum_{j=\ell}^r \lambda_j - \sum_{j=\ell}^r \frac{w_j}{U_{\ell}} (\lambda_j - \lambda_{r+1}),
\end{align*}
provided that $L_{r+1}$ and $U_{\ell}$ are non-zero,  where
$$
U_{\ell} := \sum_{i=\ell}^n w_i \quad \mbox{and} \quad L_{r+1} := \sum_{i=1}^{r+1} w_i.
$$
\end{theorem}

The proof of Theorem~\ref{thm:renormalized_bounds} is given in the next two subsections.
The following immediate relaxation maybe simpler to use in practice.

\begin{corollary}
\label{2026-01-14-cor}
For any indices $\ell$ and $r$, $1 \le \ell \le r \le n-1$, we have the inequalities:
\begin{equation} 
\label{eq:cor:renormalized_upper}
\sum_{j=\ell}^r \lambda_{j+1} + \sum_{j=\ell}^r w_{j+1} (\lambda_\ell - \lambda_{j+1})  \le  \sum_{j=\ell}^r \mu_j \le \sum_{j=\ell}^r \lambda_j - \sum_{j=\ell}^r w_j (\lambda_j - \lambda_{r+1}),
\end{equation}
provided that $L_{r+1}$ and $U_{\ell}$ are non-zero.
\end{corollary}

\begin{proof}
Since both $U_{\ell}$ and $L_{r+1}$ are at most $1$, we have
\[
\frac{w_j}{U_{\ell}} \ge w_j \quad \mbox{and} \quad \frac{w_{j+1}}{L_{r+1}} \ge w_{j+1},
\]
while the differences $\lambda_\ell - \lambda_{j+1}$ and $\lambda_j - \lambda_{r+1}$ are non-negative for $j=\ell, \ldots, r$.
\end{proof}

\subsection{Proof of Theorem~\ref{thm:renormalized_bounds}: the equal weight case}
We first establish Theorem~\ref{thm:renormalized_bounds} in the specific case when all the weights are equal. 
That is, throughout this subsection, we assume
$$
w_i = \frac{1}{n}, \mbox{ for all $i=1,\ldots, n$}.
$$
Then $\mu_1, \dots, \mu_{n-1}$ are the critical points of the polynomial 
$$
p(x) := \prod_{j=1}^n (x-\lambda_j).
$$


\begin{proposition}[Equal weight bounds]
\label{lem:equal_weights}
Assume $w_i = 1/n$ for all $i$. 
Then, for any indices $\ell$ and $r$, $1 \le \ell \le r \le n-1$, we have the inequalities:
\begin{align*} 
\sum_{j=\ell}^r \lambda_{j+1} + \frac{1}{r+1} \sum_{j=\ell}^r  (\lambda_\ell - \lambda_{j+1})  \le  \sum_{j=\ell}^r \mu_j \le \sum_{j=\ell}^r \lambda_j - \frac{1}{n-\ell+1}  \sum_{j=\ell}^r (\lambda_j - \lambda_{r+1}).
\end{align*}
\end{proposition}

\begin{proof}
Suppose first that the poles are distinct. This implies that 
$$
 \lambda_1 > \mu_1 > \lambda_2 > \mu_2 > \dots > \mu_{n-1} > \lambda_n
$$
and thus each $\mu_k$ depends smoothly on each 
$\lambda_j$. Differentiating the identity $f(\mu_k) = 0$ with respect to $\lambda_j$ yields:
\[
0=\frac{\partial}{\partial \lambda_j} f(\mu_k)
= f'(\mu_k)\,\frac{\partial \mu_k}{\partial \lambda_j}+\frac{\partial}{\partial \lambda_j}\Big(\frac1{\mu_k - \lambda_j}\Big)
= f'(\mu_k)\,\frac{\partial \mu_k}{\partial \lambda_j}+\frac1{(\mu_k - \lambda_j)^2}.
\]
Hence
\begin{align*}
\frac{\partial \mu_k}{\partial \lambda_j}
=\frac{(\mu_k - \lambda_j)^{-2}}{\sum_{i=1}^n (\mu_k - \lambda_i)^{-2}}>0.
\end{align*}
Thus every \(\mu_k\) is (strictly) increasing in every \(\lambda_j\).

We first establish the upper bound in Proposition~\ref{lem:equal_weights}. Fix the poles $\lambda_\ell, \dots, \lambda_r$, and $\lambda_{r+1}$. Let the poles $\lambda_1, \dots, \lambda_{\ell-1}$ increase to $\infty$. This causes 
the zeros $\mu_1$, $\dots$,$\mu_{\ell-1}$ to also increase to infinity, thus eliminating them from the problem. 
Increase the poles $\lambda_{r+2}, \dots, \lambda_n$ until they coincide with $\lambda_{r+1}$. All this causes the remaining zeros $\mu_\ell, \dots, \mu_{r}$ to increase to values $\mu^*_\ell, \dots, \mu^*_{r}$, while $\mu_{r+1}, \dots, \mu_{n-1}$ increase until they coincide with $\lambda_{r+1}$. In this way, 
$\mu_\ell^*, \dots, \mu_{r}^*, \underbrace{\lambda_{r+1}, \dots, \lambda_{r+1}}_{n-r-1 \text{ times}}$ are the critical points of 
\[
\tilde{p}(x) := (x - \lambda_{r+1})^{n-r} \prod_{j=\ell}^r (x - \lambda_j) =: (x - \lambda_{r+1})^{n-r} q(x).
\]
Differentiating, gives
\begin{align*}
\tilde p'(x)&=(x-\lambda_{r+1})^{n-r-1}\,\big((n-r)\,q(x)+(x-\lambda_{r+1})\,q'(x)\big) \\
&= (x-\lambda_{r+1})^{n-r-1}\,\Big((n-r)\,\Big(x^{r-\ell+1} - \Big(\sum_{j=\ell}^r \lambda_j \Big) x^{r-\ell}+ \cdots \Big) \\
&+(x-\lambda_{r+1})\,\Big((r-\ell+1)x^{r-\ell} - (r-\ell)\Big(\sum_{j=\ell}^r \lambda_j \Big) x^{r-\ell-1}+ \cdots \Big) \Big) \\
&= (x-\lambda_{r+1})^{n-r-1}\,\Big( (n - \ell+1)x^{r-\ell+1} - s x^{r-\ell} +  \cdots  \Big) \\
&= (x-\lambda_{r+1})^{n-r-1}\, (n - \ell+1) \prod_{j=\ell}^r (x - \mu_j^*),
\end{align*}
where for brevity we let
$$
s:= (n-\ell) \sum_{j=\ell}^r \lambda_j + (r-\ell+1) \lambda_{r+1}.
$$
Thus, by the Vieta's formulas, we obtain the upper bound in Proposition~\ref{lem:equal_weights}
\begin{align*}
\sum_{j=\ell}^r \mu_j \le \sum_{j=\ell}^r \mu^*_j &=  \frac{n-\ell}{n - \ell+1} \sum_{j=\ell}^r \lambda_j + \frac{r-\ell+1}{n - \ell+1} \lambda_{r+1} \\
&= \sum_{j=\ell}^r \lambda_j - \sum_{j=\ell}^r \frac{1}{n-\ell+1} (\lambda_j - \lambda_{r+1}).
\end{align*}

The proof of the lower bound is similar. Fix the poles $\lambda_\ell, \dots, \lambda_r$, and $\lambda_{r+1}$. Let the poles $\lambda_{r+2}, \ldots, \lambda_n$ decrease to $-\infty$.
This causes the zeros $\mu_{r+1}, \dots, \mu_{n-1}$  also to decrease to $-\infty$, thus eliminating them from the problem. Decrease the  poles $\lambda_1, \dots, \lambda_{\ell-1}$ until they coincide with 
$\lambda_\ell$. All this causes the zeros $\mu_{1}, \dots, \mu_{\ell-1}$ to decrease until they coincide with $\lambda_{\ell}$, while $\mu_\ell, \dots, \mu_{r}$ decrease to values 
$\mu^{**}_\ell, \dots, \mu^{**}_{r}$. In this way, 
$\underbrace{\lambda_{\ell}, \dots, \lambda_{\ell}}_{\ell - 1 \text{ times}}, \mu_\ell^*, \dots, \mu_{r}^*$ are the critical points of the polynomial
\[
\tilde{p}(x) := (x - \lambda_{\ell})^{\ell} \prod_{j=\ell}^r (x - \lambda_{j+1}).
\]
The rest of the proof is similar to that of the upper bound.

The case when the poles are not distinct can be handled as a limiting case, using the fact that each $\mu_i$ is a continuous function of the poles. 
\end{proof}

The bounds in Proposition~\ref{lem:equal_weights} can be re-written as 
\begin{align} 
\label{2026-01-10-eq:equal_upper}
\frac{r}{r+1} \sum_{j=\ell}^r \lambda_{j+1} + \frac{r-\ell+1}{r+1} \lambda_{\ell}  \le  \sum_{j=\ell}^r \mu_j \le  \frac{n - \ell}{n-\ell+1} \sum_{j=\ell}^r \lambda_j + \frac{r-\ell+1}{n-\ell+1} \lambda_{r+1}.
\end{align}
That will be useful in the next subsection. 

\subsection{Proof of Theorem~\ref{thm:renormalized_bounds}: the general weight case} Assume first that the poles are distinct and that  the weights are rational. We will dispense with these conditions at the end. Let $w_i = k_i / N$, where $k_i \in \mathbb{Z}^+$ are positive integers and $N = \sum_{i=1}^n k_i$.
    The rational function \eqref{2016-01-13-rat} can be written as:
    \begin{equation} \label{eq:rational_integers}
      f(x) =  \sum_{i=1}^n \frac{k_i}{x-\lambda_i}.
    \end{equation}
The idea is to replace each pole $\lambda_i$ with a cluster of $k_i$ distinct poles that are close to $\lambda_i$. That is,  for each $i \in \{1, \dots, n\}$, let $\Lambda_i(\epsilon) := \{ \lambda_{i,1}(\epsilon), \dots, \lambda_{i,k_i}(\epsilon) \}$ be a set of $k_i$ distinct real numbers, ordered non-increasingly, and such that $\lim_{\epsilon \to 0} \lambda_{i,s}(\epsilon) = \lambda_i$ for all $s$.
    
    Consider the perturbed function with $N$ distinct poles $\bigcup_{i=1}^n \Lambda_i(\epsilon)$, each assigned an equal weight of $1/N$:
    \begin{equation} \label{eq:perturbed_system}
      \tilde{f}(x) := \frac{1}{N} \sum_{i=1}^n \sum_{s=1}^{k_i} \frac{1}{x-\lambda_{i,s}(\epsilon)}.
    \end{equation}
    Since all $N$ poles are distinct, the zeros of $\tilde{f}(x)$, strictly interlace the $N$ poles. Moreover,  being zeros of a polynomial equation, they are continuous with respect to its coefficients (the poles). For $\epsilon$ close to zero, denote the zeros of \eqref{eq:perturbed_system} by 
\begin{align*}
\lambda_{i,1}(\epsilon) > \tilde{\mu}_{i,1}(\epsilon) >  \ldots 
> \tilde{\mu}_{i,k_i-1}(\epsilon) > \lambda_{i,k_i}(\epsilon) > \tilde{\mu}_i(\epsilon), \mbox{ for } i = 1,\ldots, n-1
\end{align*}
and 
\begin{align*}
\tilde{\mu}_{n-1}(\epsilon) > \lambda_{n,1}(\epsilon) > \tilde{\mu}_{n,1}(\epsilon) >  \ldots 
> \tilde{\mu}_{n,k_n-1}(\epsilon) > \lambda_{n,k_n}(\epsilon).
\end{align*}
As $\epsilon$ converges to $0$, $\tilde{\mu}_{i,1}(\epsilon),  \ldots,  \tilde{\mu}_{i,k_i-1}(\epsilon)$ converge to $\lambda_i$, while $\tilde{\mu}_i(\epsilon)$ converges to $\mu_i$.  By the equal weight case, see \eqref{2026-01-10-eq:equal_upper}, we have the following upper bound
\begin{align*}
\sum_{j = \ell}^r \Big(\tilde{\mu}_j(\epsilon) + \sum_{s=1}^{k_j-1}  \tilde{\mu}_{j,s}(\epsilon) \Big)
\le  \frac{N-L}{N - L+1} \sum_{j=\ell}^r \sum_{s=1}^{k_j}  {\lambda}_{j,s}(\epsilon)+ \frac{R-L+1}{N - L+1} \lambda_{r+1,1},
\end{align*}
where
$$
N := \sum_{i=1}^n k_i, \quad L:= \sum_{i=1}^{\ell-1} k_i+1, \quad \mbox{and } R:= \sum_{i=1}^{r} k_i.
$$
In the limit, as $\epsilon$ approaches $0$, we obtain
\begin{align*}
\sum_{j = \ell}^r \Big({\mu}_j + (k_j-1)  \lambda_{j} \Big)
\le   \frac{N-L}{N - L+1} \sum_{j=\ell}^r k_j  {\lambda}_{j} + \frac{R-L+1}{N - L+1} \lambda_{r+1}.
\end{align*}
Thus, 
\begin{align*}
\sum_{j = \ell}^r {\mu}_j 
&\le   \frac{N-L}{N - L+1} \sum_{j=\ell}^r k_j  {\lambda}_{j} - \sum_{j = \ell}^r (k_j-1)  \lambda_{j} + \frac{R-L+1}{N - L+1} \lambda_{r+1} \\
&=  \sum_{j = \ell}^r  \lambda_{j} + 
\Big( \frac{N-L}{N - L+1} -1 \Big) \sum_{j=\ell}^r k_j  {\lambda}_{j} + \frac{R-L+1}{N - L+1} \lambda_{r+1} \\
&= \sum_{j = \ell}^r  \lambda_{j} + \frac{1}{N-L+1}
\Big( \sum_{j=\ell}^r k_j  {\lambda}_{r+1}  - \sum_{j=\ell}^r k_j  {\lambda}_{j} \Big) \\ 
&= \sum_{j = \ell}^r  \lambda_{j} + \frac{1}{\sum_{j=\ell}^n k_j }
 \sum_{j=\ell}^r k_j  ({\lambda}_{r+1}  -  {\lambda}_{j}) \\
&= \sum_{j = \ell}^r  \lambda_{j} + \frac{1}{U_{\ell} }
 \sum_{j=\ell}^r w_j  ({\lambda}_{r+1}  -  {\lambda}_{j}), 
\end{align*}
where in the last equality we divided the numerator and the denominator by $N$.

The derivation of the lower bound is similar.

Finally,  since the rationals are dense in $\mathbb{R}$ and the roots of a polynomial are continuous functions of the coefficients, the inequalities extend to all real non-negative weights and poles with multiplicities.  This completes the proof of Theorem~\ref{thm:renormalized_bounds}.

%
%
%

\subsection{Majorization relationships}. 
\label{2026-01-10Subsec}
The bounds in Proposition~\ref{lem:equal_weights} can be relaxed by replacing both $1/(r+1)$ and $1/(n-\ell+1)$ by $1/n$, obtaining the following corollary. What is lost in precision allows us to formulate the inequalities as majorization relationships. 

\begin{corollary}
\label{cor:relaxed_equal}
Let $\lambda_1 \ge \dots \ge \lambda_n$ be real numbers and let $\mu_1 \ge \dots \ge \mu_{n-1}$ be the critical points of
\begin{align*}
p(x) := \prod_{j=1}^n (x-\lambda_j).
\end{align*}
Then, for any $r$, $1 \le r \le n-1$, we have
\begin{align}
\label{2026-01-10-ineq}
(n-1)\sum_{j=1}^r \lambda_{j+1} + r \lambda_1 \le  n\sum_{j=1}^r \mu_j 
\le (n-1) \sum_{j=1}^r \lambda_j + r \lambda_{r+1}.
\end{align}
\end{corollary}

Let $N:= n(n-1)$. Define the vectors $\lambda^*, \mu^*, \nu^* \in \mathbb{R}^N$ as follows
\begin{align*}
\lambda^* &:= (\underbrace{\lambda_1, \ldots, \lambda_1}_{\text{$n-1$ times}}, \underbrace{\lambda_2, \ldots, \lambda_2}_{\text{$n-1$ times}}, \ldots, \underbrace{\lambda_n, \ldots, \lambda_n}_{\text{$n-1$ times}})^T, \\
\mu^* &:= (\underbrace{\mu_1, \ldots, \mu_1}_{\text{$n$ times}}, \underbrace{\mu_2, \ldots, \mu_2}_{\text{$n$ times}}, \ldots, \underbrace{\mu_{n-1}, \ldots, \mu_{n-1}}_{\text{$n$ times}})^T,
\end{align*}
and letting
\[
\nu_j := \frac{n-1}{n}\lambda_{j+1} + \frac{1}{n}\lambda_1, \quad \mbox{for all } j=1,\ldots, n-1
\]
define
\[
\nu^* := (\underbrace{\nu_1, \dots, \nu_1}_{n \text{ times}}, \underbrace{\nu_2, \dots, \nu_2}_{n \text{ times}}, \dots, \underbrace{\nu_{n-1}, \dots, \nu_{n-1}}_{n \text{ times}})^T.
\]
The proof of the next theorem is immediate from \eqref{2026-01-10-ineq},

\begin{theorem}
\label{thm:majorization_construction}
We have $\nu^* \prec \mu^* \prec \lambda^*$.
\end{theorem}

It is well-known that $y \prec x$ if and only if there exists a doubly stochastic matrix $D$, such that $y = D x$, see Theorem~II.1.10 in \cite{bhatia1997}. 
We conclude this section by constructing a doubly stochastic matrix $D$, such that $\mu^* = D \lambda^*$. 

Suppose first that the poles are distinct. 
Define the matrix $W \in \mathbb{R}^{(n-1) \times n}$ with entries:
\[
W_{kj} := \frac{\partial \mu_k}{\partial \lambda_j} = \frac{(\mu_k-\lambda_j)^{-2}}{\sum_{i=1}^n (\mu_k-\lambda_i)^{-2}} \ge 0, \quad k = 1,\dots,n-1, \, j =1,\dots,n.
\]
The row sums of $W$ are clearly $\sum_{j=1}^n W_{kj} = 1$.
For the column sums, differentiating the Vieta identity 
$$
\sum_{k=1}^{n-1} \mu_k = \frac{n-1}{n} \sum_{j=1}^n \lambda_j
$$ 
with respect to $\lambda_j$ yields 
$$
\sum_{k=1}^{n-1} W_{kj} = \frac{n-1}{n}.
$$
Let 
$$
\lambda := (\lambda_1, \dots, \lambda_n)^T, \,\,\, 
\mu := (\mu_1, \dots, \mu_{n-1})^T, \mbox{ and }
\nu := (\nu_1, \dots, \nu_{n-1})^T.
$$ 
We show that  
 $\mu = W \lambda$ holds. Indeed, from 
\[
0= p'(\mu_k) = \frac{p'(\mu_k)}{p(\mu_k)}=\sum_{j=1}^n \frac{1}{\mu_k-\lambda_j},
\]
we obtain
\[
0=\sum_{j=1}^n\frac{\mu_k-\lambda_j}{(\mu_k-\lambda_j)^2}
=\mu_k\sum_{j=1}^n\frac{1}{(\mu_k-\lambda_j)^2}-\sum_{j=1}^n\frac{\lambda_j}{(\mu_k-\lambda_j)^2}.
\]
Hence
\[
\mu_k=\frac{\sum_{j=1}^n \lambda_j(\mu_k-\lambda_j)^{-2}}{\sum_{j=1}^n (\mu_k-\lambda_j)^{-2}}
=\sum_{j=1}^n W_{kj}\,\lambda_j.
\]

We inflate $W$ to a matrix $D \in \mathbb{R}^{N \times N}$ as follows. Index the rows of $D$ by pairs $(k,m)$, where $k \in \{1, \dots, n-1\}$ and $m \in \{1, \dots, n\}$.  Index the columns of $D$ by pairs $(j, \ell)$, where $j \in \{1, \dots, n\}$ is the pole index and $\ell \in \{1, \dots, n-1\}$. Define:
\[
D_{(k,m),(j,\ell)} := \frac{1}{n-1}\,W_{kj}.
\]
We now verify that $D$ is doubly stochastic.
For the row sums, fix an index $(k, m)$:
\[
\sum_{j=1}^n \sum_{\ell=1}^{n-1} D_{(k,m),(j,\ell)} = \sum_{j=1}^n \sum_{\ell=1}^{n-1} \frac{1}{n-1} W_{kj} = \sum_{j=1}^n W_{kj} \left( \sum_{\ell=1}^{n-1} \frac{1}{n-1} \right) = \sum_{j=1}^n W_{kj} = 1.
\]
For the column sums, fix a column index $(j, \ell)$:
\[
\sum_{k=1}^{n-1} \sum_{m=1}^{n} D_{(k,m),(j,\ell)} = \sum_{k=1}^{n-1} \sum_{m=1}^{n} \frac{1}{n-1} W_{kj} = \sum_{k=1}^{n-1} \frac{n}{n-1} W_{kj} = \frac{n}{n-1} \frac{n-1}{n} = 1.
\]

 Up to a permutation, we can assume that vector $\lambda^*$ has entries $\lambda^*_{(j,\ell)} = \lambda_j$ and  vector $\mu^*$ has entries $\mu^*_{(k,m)} = \mu_k$. The $(k,m)$-th component of the product $D\lambda^*$ is:
\begin{align*}
(D\lambda^*)_{(k,m)} &= \sum_{j=1}^n \sum_{\ell=1}^{n-1} D_{(k,m),(j,\ell)} \, \lambda^*_{(j,\ell)} = \sum_{j=1}^n \sum_{\ell=1}^{n-1} \frac{1}{n-1} W_{kj} \lambda_j \\
&= \sum_{j=1}^n W_{kj} \lambda_j = \mu_k = \mu^*_{(k,m)}.
\end{align*}
This shows that $\mu^* = D\lambda^*$.

The case when the poles are not distinct, can be handled with a limiting argument and since the doubly stochastic matrices form a compact set, one can choose a convergent subsequence of the matrices $D$.

\section{Geometric interpretations and aggregate bounds for the spectrum of the principal submatrices} 
\label{sec:geometric}

%

In this section, we give two equivalent geometric interpretations of the main result, Theorem \ref{thm:renormalized_bounds}. To show their equivalence, we rely on the fundamental connection between geometric projections and polynomial equations. We include the proof of the next lemma for completeness and clarity. 

Let $A$ be an $n \times n$ Hermitian matrix with eigenvalues 
$$
\lambda_1 \ge  \dots \ge \lambda_n
$$ 
and corresponding orthonormal eigenvectors $v_1, \dots, v_n$. Let $u \in \mathbb{C}^n$ be a unit vector.
Denote by $P_{u} := I - uu^*$ the orthogonal projection onto the subspace $u^\perp$. 
The {\it compression of $A$ onto $u^\perp$} is the operator 
$$
B:= P_{{u}} A P_{{u}} |_{u^\perp}.
$$
The eigenvalues of $B$ are denoted by
$$
\mu_1 \ge \dots \ge \mu_{n-1}.
$$
It is well-known that the eigenvalues of $A$ and $B$ interlace 
$$
 \lambda_1 \ge \mu_1 \ge \lambda_2 \ge \mu_2 \ge \dots \ge \mu_{n-1} \ge \lambda_n.
$$

\begin{lemma}
\label{lem:golub_secular}
The eigenvalues of $B$ are precisely the roots of
\begin{equation} \label{eq:secular_explicit}
    \sum_{i=1}^n |\langle {u}, {v}_i \rangle|^2 \prod_{\substack{j=1 \\ j \neq i}}^n (x-\lambda_j) = 0.
\end{equation}
\end{lemma}

\begin{proof}
Let $V:=[v_1\ \cdots\ v_n]$ be the unitary matrix that diagonalizes $A$: $A=V\Lambda V^*$,
where $\Lambda=\operatorname{Diag}(\lambda_1,\dots,\lambda_n)$.

Take $U\in\mathbb C^{n\times (n-1)}$ with orthonormal columns spanning $u^\perp$.
The compression of \(A\) onto \(u^\perp\) can be expressed as $B=U^*AU\in\mathbb C^{(n-1)\times (n-1)}$.
Extend $U$ to a unitary matrix \(Q=[u\ \ U]\in\mathbb C^{n\times n}\) and observe that
\[
Q^*(xI-A)Q=
\begin{pmatrix}
 u^*(xI-A)u & *\\
* & U^*(xI-A)U
\end{pmatrix}
=
\begin{pmatrix}
u^*(xI-A)u & *\\
* & xI_{n-1}-B
\end{pmatrix}.
\]
Thus, \(\det(xI_{n-1}-B)\) is equal to the \((1,1)\)-entry of
\(\operatorname{adj}( Q^*(xI-A)Q)\). Using twice the similarity invariance of the adjugate, we obtain
\begin{align*}
\det(xI_{n-1}-B) &=(\operatorname{adj} (Q^*(xI-A)Q))_{1,1} 
= (Q^* (\operatorname{adj} (xI-A) )Q)_{1,1} \\
&= u^*\,\operatorname{adj}(xI-A)\,u = (V^*u)^*\,\operatorname{adj}(xI-\Lambda)\,V^*u \\
&= \sum_{i=1}^n |\langle u,v_i\rangle|^2\prod_{\substack{j=1 \\ j \neq i}}^n (x-\lambda_j).
\end{align*}
This completes the proof.
\end{proof}

We now formulate the geometric bounds on the eigenvalues of a compression of a Hermitian matrix onto an arbitrary hyperplane. The proof shows that this result is equivalent to Theorem~\ref{thm:renormalized_bounds}.

\begin{theorem}
\label{thm:renormalized_geometric}
Let $u \in \mathbb{C}^n$ be a unit vector.
Then, for any $\ell$ and $r$, $1 \le \ell \le r \le n-1$, the eigenvalues $\mu_1, \dots, \mu_{n-1}$ of $B$ satisfy

\begin{align*}
\sum_{j=\ell}^r \lambda_{j+1} + \sum_{j=\ell}^r \frac{|\langle {u}, {v}_{j+1} \rangle|^2}{L_{r+1}({u})} (\lambda_\ell - \lambda_{j+1}) &\le    \sum_{j=\ell}^r \mu_j \\
&\le \sum_{j=\ell}^r \lambda_j - \sum_{j=\ell}^r \frac{|\langle {u}, {v}_j \rangle|^2}{U_{\ell}({u})} (\lambda_j - \lambda_{r+1}).
\end{align*}
provided that $L_{r+1}(u)$ and $U_{\ell}(u)$ are non-zero, where 
\[
U_{\ell}({u}) := \sum_{i=\ell}^n |\langle {u}, {v}_i \rangle|^2 \quad \mbox{and} \quad L_{r+1}({u}) := \sum_{i=1}^{r+1} |\langle {u}, {v}_i \rangle|^2.
\]
\end{theorem}

\begin{proof}
(Theorem \ref{thm:renormalized_bounds} $\Rightarrow$ Theorem \ref{thm:renormalized_geometric}.)
This is immediate from Lemma~\ref{lem:golub_secular}, setting $w_i = |\langle {u}, {v}_i \rangle|^2$ for all $i$, and observing that $\sum_{i=1}^n w_i =1$ since $\|u\|=1$. 

(Theorem \ref{thm:renormalized_geometric} $\Rightarrow$ Theorem \ref{thm:renormalized_bounds}).
Consider any real numbers $\lambda_1 \ge  \dots \ge \lambda_n$ and let 
$w_1, \dots, w_n$ be non-negative weights, such that $\sum_{i=1}^n w_i = 1.$
Let $A$ be $\text{Diag}(\lambda_1, \dots, \lambda_n)$, then $v_1,\dots, v_n$ is the standard orthonormal basis and for 
 ${u} := (\sqrt{w_1}, \dots, \sqrt{w_n})^T$, we have $|\langle {u}, v_i \rangle|^2 = w_i$. Theorem \ref{thm:renormalized_bounds} follows.
\end{proof}

The  problem of bounding eigenvalues of principal submatrices corresponds to a specific choice of the projection vector 
${u}$. The proof of the next result shows it is also equivalent to Theorems~\ref{thm:renormalized_bounds} and \ref{thm:renormalized_geometric}.

\begin{theorem}
\label{thm:renormalized_submatrix}
Fix $k$, $1\le k \le n$. 
Then, for any $\ell$ and $r$, $1 \le \ell \le r \le n-1$, the eigenvalues $\mu_{k,1}, \dots, \mu_{k,n-1}$ of $A_k$ satisfy
\begin{align*}
\sum_{j=\ell}^r \lambda_{j+1} + \sum_{j=\ell}^r \frac{|v_{j+1,k}|^2}{L_{r+1}^{k}} (\lambda_\ell - \lambda_{j+1}) 
\le  \sum_{j=\ell}^r \mu_{k,j} \le \sum_{j=\ell}^r \lambda_j - \sum_{j=\ell}^r \frac{|v_{j,k}|^2}{U_{\ell}^{k}} (\lambda_j - \lambda_{r+1}),
\end{align*}
provided that $L_{r+1}^k$ and $U_{\ell}^k$ are non-zero, where
\begin{align}
\label{2026-01-12-defn}
U_{\ell}^{k} := \sum_{i=\ell}^n |v_{i,k}|^2 \quad \mbox{and} \quad L_{r+1}^{k} := \sum_{i=1}^{r+1} |v_{i,k}|^2.
\end{align}

\end{theorem}

\begin{proof}
(Theorem \ref{thm:renormalized_geometric} $\Rightarrow$ Theorem \ref{thm:renormalized_submatrix}.)
Apply Theorem \ref{thm:renormalized_geometric} with ${u} = {e}_k$, the $k$-th standard basis vector. 
The compression of $A$ onto the subspace orthogonal to ${e}_k$ is the principal submatrix $A_k$. 

(Theorem \ref{thm:renormalized_submatrix} $\Rightarrow$ Theorem \ref{thm:renormalized_geometric}.)
Consider an arbitrary unit vector ${u} \in \mathbb{C}^n$.
Take a unitary matrix $U$, such that $U{u} = {e}_1$ and define the Hermitian matrix $B = U A U^*$, having the same spectrum as $A$. The compression of $A$ onto ${u}^\perp$ has the same spectrum as the compression of $B$ onto $(U{u})^\perp = {e}_1^\perp$, and the latter is precisely the principal submatrix $B_1$.
By Theorem \ref{thm:renormalized_submatrix} applied to $B$, with $k=1$, and weights $|w_{i,1}|$ given by its eigenvectors ${w}_i = U{v}_i$:
\[
w_{i,1} = \langle {e}_1, {w}_i \rangle = \langle {e}_1, U{v}_i \rangle = \langle U^*{e}_1, {v}_i \rangle 
= \langle {u}, {v}_i \rangle,
\]
one deduces Theorem \ref{thm:renormalized_geometric}.
\end{proof}

We are now ready to state the result about the aggregate bounds on the eigenvalues of the $(n-1) \times (n-1)$ principal submatrices.

\begin{theorem}
\label{thm:aggregate_renormalized} 
For any $\ell$ and $r$, $1 \le \ell \le r \le n-1$, the eigenvalues of $A_k$, $k=1,\ldots, n$, satisfy
\begin{align*} \label{eq:global_renormalized_upper}
 n \sum_{j=\ell}^r \lambda_{j+1} + \sum_{j=\ell}^r (\lambda_l - \lambda_{j+1}) \Psi_{j+1}(r)  
 &\le  \sum_{k=1}^n \sum_{j=\ell}^r \mu_{k,j} \\
 & \le n \sum_{j=\ell}^r \lambda_j - \sum_{j=\ell}^r (\lambda_j - \lambda_{r+1}) \Omega_j(\ell),
\end{align*}
provided that $\Psi_{j+1}(r)$ and $\Omega_j(\ell)$ are well-defined for all $j = \ell, \ldots, r$, where
\[
\Omega_j(\ell) := \sum_{k=1}^n \frac{|v_{j,k}|^2}{U_{\ell}^{k}} \quad \text{and} \quad \Psi_j(r) := \sum_{k=1}^n \frac{|v_{j,k}|^2}{L_{r+1}^{k}},
\]
with $U_{\ell}^{k}$ and $L_{r+1}^{k}$ defined in \eqref{2026-01-12-defn}.
\end{theorem}

\begin{proof}
Summing the upper bounds in Theorem \ref{thm:renormalized_submatrix} over $k=1, \dots, n$, one obtains
\begin{align*}
\sum_{k=1}^n \sum_{j=\ell}^r \mu_{k,j} &\le  n \sum_{j=\ell}^r \lambda_j - \sum_{k=1}^n \sum_{j=\ell}^r \frac{|v_{j,k}|^2}{U_{\ell}^{k}} (\lambda_j - \lambda_{r+1})  \\
 &=n \sum_{j=\ell}^r \lambda_j  - \sum_{j=\ell}^r (\lambda_j - \lambda_{r+1}) \Big( \sum_{k=1}^n \frac{|v_{j,k}|^2}{U_{\ell}^{k}} \Big).
\end{align*}
The derivation of the lower bound is similar.
\end{proof}


Relaxing the bounds a bit, allows us to remove the conditions on the coefficients $\Psi_{j+1}(r)$ and $\Omega_j(\ell)$.

\begin{theorem}
\label{thm:aggregate_bound}
Let $A$ be any  $n \times n$ Hermitian matrix. For any $\ell$ and $r$, $1 \le \ell \le r \le n-1$, the eigenvalues of $A_k$, $k=1,\ldots, n$, satisfy
\begin{align*} 
\label{eq:global_upper}
(r-\ell+1)\lambda_{\ell} + (n-1)\sum_{j=\ell}^r \lambda_{j+1}   \le \sum_{k=1}^n \sum_{j=\ell}^r \mu_{k,j} \le (n-1)\sum_{j=\ell}^r \lambda_j + (r-\ell+1)\lambda_{r+1}.
\end{align*}
\end{theorem}

\begin{proof}
Take a sequence $\{A^m\}$ of Hermitian matrices converging to $A$, such that for each $A^m$, the expressions $\Psi_{j+1}(r)$ and $\Omega_j(\ell)$ are well defined for all $j = \ell, \ldots, r$. (That is no denominator is zero.)
Since for $j = \ell, \ldots, r$, we have 
\[
\Omega_j(\ell) = \sum_{k=1}^n \frac{|v_{j,k}|^2}{U_{\ell}^{k}} \ge \sum_{k=1}^n |v_{j,k}|^2 = 1,
\]
the upper bound in Theorem~\ref{thm:aggregate_renormalized}, for the eigenvalues of $A^m$ and its $(n-1) \times (n-1)$ principal submatrices, can be relaxed to  
$$
\sum_{k=1}^n \sum_{j=\ell}^r \mu_{k,j} \le n \sum_{j=\ell}^r \lambda_j - \sum_{j=\ell}^r (\lambda_j - \lambda_{r+1})  = (n-1) \sum_{j=\ell}^r \lambda_j + (r-\ell+1) \lambda_{r+1}.
$$
(For simplicity, the dependence on $m$ is understood implicitly.)
Take the limit as $m$ approaches infinity to conclude. The justification of the lower bound is similar.
\end{proof}

Clearly, when  $\ell=r=j$, the bounds in Theorem~\ref{thm:aggregate_bound} reduce to 
\begin{align*} 
(n-1) \lambda_{j+1} + \lambda_{j}  \le \sum_{k=1}^n  \mu_{k,j} \le (n-1)\lambda_j + \lambda_{j+1}
\end{align*}
which are Thompson's bounds in Theorem~\ref{thm:spectral_bounds}.

\begin{corollary}
\label{2026-01-15-corBody}
For any $\ell$ and $r$, $1 \le \ell \le r \le n-1$, the following inequalities hold: 
$$
(n-\ell)\lambda_\ell  + (n-1)\sum_{j=\ell+1}^{r} \lambda_j + r\lambda_n
\le \sum_{k=1}^n  \sum_{j=\ell}^r  \mu_{k,j} 
\le (n-\ell)\lambda_1  + (n-1)\sum_{j=\ell+1}^{r} \lambda_j + r\lambda_{r+1}.
$$
\end{corollary}

\begin{proof}
For the lower bound,  express the sum over $[\ell, r]$ as the difference between the sum over $[\ell, n-1]$ and $[r+1, n-1]$.
Applying the lower bound of Theorem~\ref{thm:aggregate_bound}  to $[\ell, n-1]$ and the upper bound of Theorem~\ref{thm:aggregate_bound} to $[r+1, n-1]$ yields:
\begin{align*}
\sum_{k=1}^n \sum_{j=\ell}^r  \mu_{k,j} &= \sum_{k=1}^n \sum_{j=\ell}^{n-1}  \mu_{k,j} - \sum_{k=1}^n \sum_{j=r+1}^{n-1}  \mu_{k,j} \\
&\ge \Big[ (n-1)\sum_{j=\ell}^{n-1} \lambda_{j+1} + (n-\ell)\lambda_\ell \Big] - \Big[ (n-1)\sum_{j=r+1}^{n-1} \lambda_j + (n-r-1)\lambda_n \Big] \\
&= (n-\ell)\lambda_\ell + r\lambda_n + (n-1)\sum_{j=\ell+1}^r \lambda_j.
\end{align*}

For the upper bound, express the sum over $[\ell, r]$ as the difference between the sum over $[1, r]$ and $[1, \ell-1]$.
Applying the upper bound of Theorem~\ref{thm:aggregate_bound}  to $[1, r]$ and the lower bound of Theorem~\ref{thm:aggregate_bound}  to $[1, \ell-1]$ yields:
\begin{align*}
\sum_{k=1}^n \sum_{j=\ell}^r  \mu_{k,j} &= \sum_{k=1}^n \sum_{j=1}^{r}  \mu_{k,j} - \sum_{k=1}^n \sum_{j=1}^{\ell-1}  \mu_{k,j} \\
&\le \Big[ (n-1)\sum_{j=1}^{r} \lambda_j + r\lambda_{r+1} \Big] - \Big[ (\ell-1)\lambda_1 + (n-1)\sum_{j=1}^{\ell-1} \lambda_{j+1} \Big] \\
&= (n-\ell)\lambda_1  + (n-1)\sum_{k=\ell+1}^r \lambda_k + r\lambda_{r+1},
\end{align*}
concluding the proof.
\end{proof}

\begin{remark} 
Surprisingly a bound in Corollary~\ref{2026-01-15-corBody} may be better than the corresponding bound in Theorem~\ref{thm:aggregate_bound} and vice-versa. To see that let us calculate the difference between the upper bound in the theorem  and the upper bound in the corollary, both applied to the interval $[\ell, r]$:
\begin{align*}
&\quad \text{UB}_{\text{Thm}} - \text{UB}_{\text{Cor}} \\
&= \Big[ (n-1)\sum_{j=\ell}^r \lambda_j + (r-\ell+1)\lambda_{r+1} \Big] - \Big[ (n-\ell)\lambda_1  + (n-1)\sum_{j=\ell+1}^{r} \lambda_j + r\lambda_{r+1} \Big] \\
&= (n-1)\lambda_\ell - (\ell-1)\lambda_{r+1} - (n-\ell)\lambda_1 \\
&= (\ell-1)(\lambda_\ell - \lambda_{r+1}) -(n-\ell)(\lambda_1 - \lambda_\ell).
\end{align*}
If $\ell = 1$, the last difference is zero. Otherwise, it could have any sign.
Analogously, we have $\text{LB}_{\text{Thm}} - \text{LB}_{\text{Cor}} = 0$, whenever $r=n$.

If one repeats the same idea to the bounds in Corollary~\ref{2026-01-15-corBody}, the result is bounds that are looser than those in Theorem~\ref{thm:aggregate_bound}.
\end{remark}

\section{Aggregate majorization theorems for principal submatrices}

The goal of this section is to establish Theorem~\ref{thm:hierarchy}. 
It establishes a complete hierarchical relationship between the eigenvalues of all principal submatrices of size $m$ and the eigenvalues of all principal submatrices of size $k$, where $1 \le k \le m \le n$. As a consequence, we recover the chain of the well-known Szasz's inequalities relating products of principal minors.

Let $X_m(A)$, $1 \le m \le n$, be the vector containing the eigenvalues of all $m \times m$ principal submatrices of the Hermitian matrix
$A$, counting repetitions. For example, $X_n(A) = \lambda(A)$ and $X_1(A) = \mbox{diag\,}(A)$, the vector of the diagonal entries of $A$. Below, we abuse the union notation to denote the concatenation of vectors. 

\begin{theorem}[Spectral Hierarchy]
\label{thm:hierarchy}
Let $A$ be an $n \times n$ Hermitian matrix. For any integers $1 \le k \le m \le n$, we have
\begin{align} 
\label{eq:general_hierarchy}
    \underbrace{X_m(A) \cup \dots \cup X_m(A)}_{\binom{m-1}{k-1} \text{ times}} \succ \underbrace{X_k(A) \cup \dots \cup X_k(A)}_{\binom{n-k}{m-k} \text{ times}}.
\end{align}
\end{theorem}

Several special cases are worth mentioning. 

 \noindent
$\bullet$ When $m=n$, we have
    \[
    \underbrace{\lambda(A) \cup \dots \cup \lambda(A)}_{\binom{n-1}{k-1} \text{ times}} \succ X_k(A), \mbox{ for all } k=1,\ldots, n.
    \]
    In particular, when $k=1$, one gets Schur's theorem $\lambda(A) \succ \mbox{diag\,}(A).$

 \noindent
$\bullet$ When $m=k+1$, we have
    \[
    \underbrace{X_{k+1}(A) \cup \dots \cup X_{k+1}(A)}_{k \text{ times}} \succ \underbrace{X_k(A) \cup \dots \cup X_k(A)}_{n-k \text{ times}}, \mbox{ for all } k=1,\ldots, n-1.
    \]
    
 \noindent
$\bullet$ When $m+k=n+1$, observe that the repetition factors on both sides of \eqref{eq:general_hierarchy} are equal:
    \[
    \binom{m-1}{k-1} = \binom{m-1}{m-k} = \binom{(m+k-1)-k}{m-k} = \binom{n-k}{m-k}.
    \]
    Thus, we can perform a cancellation, yielding a direct majorization between the vectors:
    \[
    X_m(A) \succ X_k(A), \mbox{ for all } 1 \le k \le m \le n \mbox{ with } m+k=n+1.
    \]
    The case $k=1$, forcing $m=n$, again recovers Schur's theorem.
    
 \noindent
$\bullet$ Let $P_m(A)$ denote the product of all $m \times m$ principal minors of $A$, $1 \le m \le n$. For example $P_n(A) = \det (A)$ and $P_1(A) = A_{1,1}\cdots A_{n,n}$, the product of the diagonal entries of $A$.  Since the elementary symmetric polynomials are Schur concave in the non-negative orthant, see Example~II.3.16 in \cite{bhatia1997}, one has that if $x \succ y$, where $x, y \in \mathbb{R}^N_+$, then $\prod_{i=1}^N x_i \le \prod_{i=1}^N y_i$.
Thus, \eqref{eq:general_hierarchy} implies that for positive semidefinite matrix $A$, we have
$$
P_{m}(A)^{\binom{m-1}{k-1}} \le P_k(A)^{\binom{n-k}{m-k}}, \mbox{ for all } 1 \le k \le m \le n.
$$
Raising both sides to the power $1/Z$, where 
$$
Z:=\binom{m-1}{k-1}\binom{n-1}{m-1} = \binom{n-k}{m-k}\binom{n-1}{k-1},
$$ 
leads to 
\begin{equation}
    P_{m}(A)^{1/\binom{n-1}{m-1}} \le P_k(A)^{1/\binom{n-1}{k-1}}, \mbox{ for all } 1 \le k \le m \le n.
\end{equation}
These are the well-known Szasz's inequalities, see \cite[Theorem 7.8.11]{horn2013}.

Let
$$
Y_{n,m} (A):=  \underbrace{\lambda(A) \cup \dots \cup \lambda(A)}_{\binom{n-1}{m-1} \text{ times}}, \mbox{ for } 1 \le m \le n.
$$
In what follows we assume that both $X_m$ and $Y_{n,m}$ are ordered non-increasingly. 
Note that 
$$
Y_{n,n}(A) =X_n(A) = \lambda(A) \mbox{ and } Y_{n,n-1}(A) = \underbrace{\lambda(A) \cup \dots \cup \lambda(A)}_{n-1 \text{ times}}.
$$

\begin{lemma}[Nodal majorization]
\label{cor:nodal_verification}
We have 
\begin{align*}
    \sum_{i=1}^{rn} (X_{n-1})_i \le \sum_{i=1}^{rn} (Y_{n,n-1})_i, \quad \mbox{ for all } r=1,\ldots, n-1.
\end{align*}
The inequality becomes equality when $r=n-1$.
\end{lemma}

\begin{proof}
This follows from the upper bound in Theorem~\ref{thm:aggregate_bound}, applied with $\ell=1$:
\[
\sum_{i=1}^{rn} (X_{n-1})_i
= \sum_{k=1}^n \sum_{j=1}^r \mu_{k,j} 
\le (n-1)\sum_{j=1}^r \lambda_j + r\lambda_{r+1}
= \sum_{i=1}^{rn} (Y_{n,n-1})_i.
\]
When $r=n-1$, we have
\begin{align*}
\sum_{i=1}^{(n-1)n} (X_{n-1})_i
&= \sum_{k=1}^n \sum_{j=1}^{n-1} \mu_{k,j}
= \sum_{k=1}^n \mbox{tr\,} (A_k)
= (n-1) \mbox{tr\,} (A) \\
&= (n-1)\sum_{j=1}^{n-1} \lambda_j + (n-1)\lambda_{n}
= \sum_{i=1}^{(n-1)n} (Y_{n,n-1})_i.
\end{align*}
This establishes a majorization type relationship at the nodal indices  $p=rn$. 
\end{proof}

Before proving Theorem~\ref{thm:hierarchy}, we establish two structural lemmas. The proof of the first shows that for the specific case of $(n-1) \times (n-1)$ submatrices, the majorization condition  above at the nodal indices ($p=rn$) is sufficient. The second lemma establishes an inductive link between the submatrix sizes.

\begin{lemma}
\label{lem:nodal_reduction}
For $n \ge 2$, we have $Y_{n,n-1} (A)\succ X_{n-1}(A)$, for all $n \times n$ 
Hermitian matrices $A$.
\end{lemma}

\begin{proof}
In this proof omit the argument $A$ for brevity.
The ordered vector $Y_{n,n-1}$ is constructed in blocks of equal entries of sizes $n-1$. The entries within block $k$ are equal to $\lambda_k$. The indexes of the entries in block $k$ range from $(k-1)(n-1)+1$ up to $k(n-1)$.
In contrast, $X_{n-1}$ is constructed in layers of size $n$. The entries in the $k$-th layer are $\mu_{1,k}, \ldots, \mu_{n,k}$, 
the $k$-th largest eigenvalues of the $(n-1) \times (n-1)$ principal submatrices $A_1,\ldots, A_n$.
The indexes of the entries in layer $k$  range from $(k-1)n+1$ up to $kn$. By the interlacing, these elements satisfy 
 $\lambda_{k} \ge \mu_{i,k} \ge \lambda_{k+1}$ for all $i=1,\ldots, n$.

We analyze the behaviour of the differences of partial sums  
\[
\Delta_p = \sum_{i=1}^p (Y_{n,n-1})_i - \sum_{i=1}^p (X_{n-1})_i.
\]
within each interval $p \in \{(k-1)n + 1, \dots, kn\}$.

First, consider the range $(k-1)n+1 \le p \le k(n-1)$. Since $(k-1)(n-1) \le (k-1)n$, in this subinterval, $p$ is within the $k$-th block of $Y_{n,n-1}$, so 
$(Y_{n,n-1})_p = \lambda_k$. Since $(X_{n-1})_p$ belongs to the $k$-th layer, we have $\lambda_k \ge (X_{n-1})_p$. The increment in the partial sum differences is $\Delta_p - \Delta_{p-1} = (Y_{n,n-1})_p - (X_{n-1})_p = \lambda_k - (X_{n-1})_p \ge 0$. Thus, $\Delta_p$ is non-decreasing starting from the node $(k-1)n$. By Lemma~\ref{cor:nodal_verification}, we have 
$\Delta_{(k-1)n} \ge 0$, so the values remain non-negative throughout this range.

Second, consider the remaining subrange $k(n-1) +1 \le p \le kn$. Here, $p$ has entered the $(k+1)$-th block of $Y_{n,n-1}$, so 
$(Y_{n,n-1})_p = \lambda_{k+1}$. Since $(X_{n-1})_p$ remains in the $k$-th layer, the interlacing lower bound implies 
$(X_{n-1})_p \ge \lambda_{k+1}$. Consequently, the increment is $\Delta_p - \Delta_{p-1} = \lambda_{k+1} - (X_{n-1})_p \le 0$. This implies that $\Delta_p$ is  non-increasing in this subinterval as it approaches the next node $kn$.

Combining these observations,  within each interval $p \in \{(k-1)n + 1, \dots, kn\}$, the sequence $\Delta_p$ is unimodal: it first rises and then falls. By Lemma~\ref{cor:nodal_verification} the sequence starts and ends at a non-negative value  
$\Delta_{(k-1)n}, \Delta_{kn} \ge 0$. Thus, $\Delta_p \ge 0$ for all $p \in \{(k-1)n + 1, \dots, kn\}$.

The fact that  $\Delta_{(n-1)n} = 0$ follows from Lemma~\ref{cor:nodal_verification}.
\end{proof}

We need two simple properties of majorization.  Let $u^{i}, v^i \in \R^{n_i}$, $i=1,\ldots, k$, be two collections of vectors. If $u^{i} \succ v^{i}$ for all $i = 1, \dots, k$, then 
\begin{align}
\label{2026-01-12-succ}
u^{1} \cup \cdots \cup u^k \succ v^{1} \cup \cdots \cup v^k.
\end{align}
We call this the {\it scaling property}. 
If, $n_i = N, u^i = u$ and $v^i = v$ for all $i=1,\ldots, k$, then \eqref{2026-01-12-succ} is equivalent to $u \succ v$. We call this the {\it cancellation property}.

\begin{lemma}
\label{lem:induction_principal}
Let $n \ge 2$. Suppose that $Y_{n-1, m}(B) \succ X_m(B)$ for $1 \le m \le n-1$ and for all $(n-1) \times (n-1)$ 
Hermitian matrices $B$. Then $Y_{n, m}(A) \succ X_{m}(A)$ for all $n \times n$ 
Hermitian matrices $A$ and all $1 \le m \le n$.
\end{lemma}

\begin{proof}
By assumption,  we have $Y_{n-1, m}(A_k) \succ X_m(A_k)$, for any $1 \le m \le n-1$ and $k=1,\ldots, n$. By the scaling property, we obtain
\begin{equation} \label{eq:induct_sum}
   Y_{n-1, m}(A_1) \cup \cdots \cup Y_{n-1, m}(A_n) \succ X_m(A_1) \cup \cdots \cup X_m(A_n).
\end{equation}

On the left-hand side, each term $Y_{n-1, m}(A_k)$ consists of $\binom{n-2}{m-1}$ copies of $\lambda(A_k)$. 
The concatenation  $\bigcup_{k=1}^n \lambda(A_k)$ constitutes precisely the vector $X_{n-1}(A)$. Thus, 
 the aggregate left-hand side represents $\binom{n-2}{m-1}$ copies of the vector $X_{n-1}(A)$.

On the right-hand side, the union represents the collection of eigenvalues of all $m \times m$ principal submatrices found within the matrices $A_1, \dots, A_n$. Consider an arbitrary $m \times m$ submatrix $M$ of $A$. This matrix $M$ appears as a submatrix of an $A_k$ if and only if the single index deleted to form $A_k$ is chosen from the $n-m$ indices not used by $M$. Consequently, each distinct $M$ is counted exactly $n-m$ times. The aggregate right-hand side is therefore equivalent to $(n-m)$ copies of the vector $X_m(A)$.

Substituting back into \eqref{eq:induct_sum}, we obtain:
\[
 \underbrace{X_{n-1}(A) \cup \cdots \cup X_{n-1}(A)}_{\binom{n-2}{m-1} \text{ times}} \succ  \underbrace{X_m(A) \cup \cdots \cup X_m(A) }_{n-m \text{ times}}.
\]
By Lemma~\ref{lem:nodal_reduction}, we have  
$$
 \underbrace{\lambda(A) \cup \dots \cup \lambda(A)}_{n-1 \text{ times}} \succ X_{n-1}(A).
$$
 Scaling this relation by the factor $\binom{n-2}{m-1}$ and applying transitivity, we get:
\[
 \underbrace{\lambda(A) \cup \dots \cup \lambda(A)}_{\binom{n-2}{m-1} (n-1)  \text{ times}} 
\succ  \underbrace{X_m(A) \cup \cdots \cup X_m(A) }_{n-m \text{ times}}.
\]
Finally, applying the cancellation property, we divide the repetition counts by $(n-m)$. The resulting number of copies of $\lambda(A)$ is:
\[
\frac{n-1}{n-m} \binom{n-2}{m-1} = \frac{n-1}{n-m} \cdot \frac{(n-2)!}{(m-1)!(n-m-1)!} = \binom{n-1}{m-1}.
\]
This count matches the definition of $Y_{n,m}(A)$. Therefore, $Y_{n,m}(A) \succ X_m(A)$.
\end{proof}

\begin{theorem}
\label{thm:general_majorization}
We have 
\[
Y_{n,m}(A) \succ X_m(A), \mbox{ for all } 1 \le m \le n
\]
and any $n \times n$ Hermitian matrix $A$. 
\end{theorem}

\begin{proof}
The validity of the boundary cases $m=1$ and $m=n$ is clear. For $m=1$, the statement $Y_{n,1}(A) \succ X_1(A)$ is precisely Schur's theorem \eqref{2026-01-13-schur}. For $m=n$, we have  $Y_{n,n}(A) = X_n(A)$.
These boundary cases are sufficient to verify the theorem for low dimensions. Specifically, for $n=1$ (with $m=1$) and $n=2$ (with $m=1, 2$), the result holds immediately.

For $n \ge 3$, to address the remaining cases, we use induction on the dimension of the matrix. If we assume that $Y_{n-1, m}(B) \succ X_m(B)$ holds for $1 \le m \le n-1$ and for all $(n-1) \times (n-1)$ 
Hermitian matrices $B$, then Lemma \ref{lem:induction_principal} ensures that  $Y_{n, m}(A) \succ X_{m}(A)$ for all $n \times n$  Hermitian matrices $A$ and all $1 \le m \le n$.

The base case of the  induction is $n=2$, when we trivially have $Y_{1, 1}(B) \succ X_1(B)$ for and all $1 \times 1$ matrices $B$.
\end{proof}

\begin{proof}[Proof of Theorem~\ref{thm:hierarchy}]
Let $M$ be an $m \times m$ principal submatrix of $A$. Applying Theorem \ref{thm:general_majorization} to $M$ with subsize $k$, $1 \le k \le m$, we have:
    \begin{equation} \label{eq:local_hierarchy}
        \underbrace{\lambda(M) \cup \dots \cup \lambda(M)}_{\binom{m-1}{k-1} \text{ times}} = Y_{m,k}(M)\succ X_k(M),
    \end{equation}
    where $X_k(M)$ is the vector of eigenvalues of all $k \times k$ submatrices of $M$.
    
    Take the union of \eqref{eq:local_hierarchy} over all  $m \times m$ principal submatrices $M$ of $A$.   
    On the left-hand side of that union we have precisely the vector $X_m(A)$ repeated $\binom{m-1}{k-1}$ times.
The right-hand side of the union collects every $k \times k$ submatrix contained within every $m \times m$ parent $M$.  Let $Q$ be a specific $k \times k$ principal submatrix of $A$. $Q$ is a submatrix of an $m \times m$ parent $M$ if and only if the index set of $M$ contains the index set of $Q$. Since $Q$ fixes $k$ indices, we must choose the remaining $m-k$ indices from the available $n-k$ indices to form $M$.
    Thus, there are exactly $\binom{n-k}{m-k}$ such parent matrices $M$. Consequently, the eigenvalues of each $Q$ appear $\binom{n-k}{m-k}$ times in the aggregation, establishing \eqref{eq:general_hierarchy}.
\end{proof}

{\bf Acknowledgments: } The first author is grateful to the Technical University of Vienna and particularly to the department of Variational Analysis, Dynamics and Operations Research with head Dr. Aris Daniilidis for their warm hospitality during the work on this paper.

\end{document}